\documentclass[12pt]{amsart}

\usepackage{amssymb}
\usepackage{graphicx}
\usepackage{epstopdf}
\usepackage{caption}

\theoremstyle{plain}
\newtheorem{thm}{Theorem}[section]

\theoremstyle{definition}
\newtheorem{defn}[thm]{Definition}
\newtheorem{prop}[thm]{Proposition}
\newtheorem{rem}[thm]{Remark}
\newtheorem{lem}[thm]{Lemma}
\newtheorem{cor}[thm]{Corollary}
\newtheorem{ex}[thm]{Example}

\newcommand{\Z}{\mathbb{Z}}

\DeclareMathOperator{\LLS}{LLS}
\DeclareMathOperator{\SL}{SL}
\DeclareMathOperator{\sign}{sign}

\title{Generalized Perron Identity for broken lines}

\author{Oleg Karpenkov, Matty van-Son}
\date{30 December 2017}

\address{Oleg Karpenkov\\
University of Liverpool\\
Mathematical Sciences Building\\
Liverpool L69 7ZL, United Kingdom
} \email{karpenk@liv.ac.uk}

\address{Matty van-Son\\
	University of Liverpool\\
	Mathematical Sciences Building\\
	Liverpool L69 7ZL, United Kingdom
} \email{sgmvanso@liverpool.ac.uk}

\thanks
{
O.~Karpenkov is partially supported by EPSRC grant EP/N014499/1 (LCMH)
}


\begin{document}
\begin{abstract}
In this paper we generalize the Perron Identity for Markov minima.
We express the values of binary quadratic forms with positive discriminant
in terms of continued fractions associated to broken lines passing through the points
where the values are computed.
\end{abstract}

\maketitle
\input epsf
\tableofcontents

\section*{Introduction}
Consider a binary quadratic form $f$ with positive discriminant $\Delta(f)$.
In this paper we give a geometric interpretation and generalization of the {\it Perron Identity} relating
the minimal value of $|f|$ at integer points except the origin and their corresponding continued fractions:
\begin{equation} \label{per eq}
		\min\limits_{\Z^2\setminus \{(0,0)\}}\big|f\big|=\inf\limits_{i\in \Z}\bigg(\frac{\sqrt{\Delta(f)}}{a_i+[0;a_{i+1}:a_{i+2}:\ldots]+[0;a_{i-1}:a_{i-2}:\ldots]}\bigg).
		\end{equation}
Here $[a_0;a_1:\ldots]$ and $[0;a_{-1}:a_{-2}:\ldots]$ 
are regular continued fractions of the slopes of linear factors of corresponding reduced linear forms. Recall that a continued fraction is regular if all its elements are non negative.
We discuss this in more detail further in Section~\ref{Basic notions and definitions}.

\vspace{2mm}

The Perron Identity was shown by A.~Markov in his paper on minima of binary quadratic forms and the Markov spectrum below 3 in~\cite{mar2}.
The statement holds for the entire Markov spectrum (see, e.g., the books by O.~Perron~\cite{per1}, and T.~Cusick and M.~Flahive~\cite{cus1}).
Recently Markov numbers were used in relation to Federer-Gromov's stable norm, 
(\cite{Fock2007,Veselov2017}).
There is not much known about higher dimensional analogue of Markov spectrum.
It is believed to be discrete (which is equivalent to Oppenheim conjecture on best approximations,
see in Chapter 18 of~\cite{oleg1}). Various values of three-dimensional Markov spectrum were constructed 
by H.~Davenport in~\cite{Davenport1938,Davenport1938a,Davenport1939}.

\vspace{2mm}
	
In this paper we show the geometric interpretation of the Perron Identity
in terms of sails of the form (Remark~\ref{RefGeometry})
and generalize this expression in the spirit of integer geometry. This establishes a relationship between non-regular continued fractions and the values of the corresponding binary quadratic form
at any point on the plane (Theorem~\ref{MainTheorem} and Corollary~\ref{MainCorollary}).
The result of this paper is based on recent results of the first author in
geometric theory of continued fractions for arbitrary broken lines, see~\cite{oleg3,oleg2,oleg4,oleg1}.

\vspace{2mm}
	
{\noindent
{\bf Organization of the paper.}
We start in Section~\ref{Basic notions and definitions} with necessary definitions and background.
We discuss reduced forms, LLS sequences, and formulate the classical Perron Identity.
In Section~\ref{thm sec} we formulate and prove the Generalized Perron Identity for finite broken lines.
Finally in Section~\ref{Generalized Perron identity for asymptotic infinite broken lines}
we prove the Generalized Perron Identity for infinite broken lines, and discuss the relation with
the classical the Perron Identity.
}
	
\vspace{2mm}	
	
{\noindent
{\bf Acknowledgement.}
The first author is partially supported by EPSRC grant EP/N014499/1 (LCMH).
}

\section{Basic notions and definitions}
\label{Basic notions and definitions}

In this section we give necessary notions and definitions.
We start in Subsection~\ref{Markov minima and Markov spectrum}
with classical definitions of Markov minima and Markov spectrum.
Further in Subsection~\ref{Reduced forms, and LLS-sequences}
we discuss reduced forms of quadratic binary forms with positive discriminant.
In Subsection~\ref{Classical Perron Identity} we discuss the classical Perron Identity.
Finally in Subsection~\ref{LLS sequences for broken lines}
we introduce LLS sequences for broken lines, which is the central notion in the formulation of the main results.

\subsection{Markov minima and Markov spectrum}
\label{Markov minima and Markov spectrum}
	
Let $f$ be a binary quadratic form with positive discriminant.
Recall that in this case $f$ is decomposable into two real factors, namely
\[
f(x,y)=(ax-b y)(cx-d y),
\]
for some real numbers $a$, $b$, $c$, and $d$.
The discriminant of this form is
\[
\Delta(f)=(ad-bc)^2.
\]
	
The {\it Markov minimum} of the form $f$ is the following number:
\[
m(f)=\min\limits_{\Z^2\setminus\{(0,0)\}}|f|.
\]
The set of all possible values of $\Delta(f)/m(f)$ is called {\it Markov Spectrum}.
(Note that $\Delta(f)/m(f)$ is invariant under multiplication of the form $f$ by a non-zero scalar.)
The spectrum below 3 correspond to special forms with integer coefficients,
we refer an interested reader to an excellent book~\cite{cus1} by T.~Cusick and M.~Flahive on Markov spectrum and related subjects.

\subsection{Reduced forms, and LLS-sequences}
\label{Reduced forms, and LLS-sequences}

It is clear that $m(f)$ is invariant under the action of the group of $\SL(2,\Z)$.
Therefore in order to study the Markov spectrum one can restrict to
so called {\it reduced forms} which are simple to describe.
There are several ways to pick reduced forms, 
although the algorithmic part is rather similar to all of them, it is a subject of a Gauss reduction theory (see, e.g.,~\cite{Lewis1997}, \cite{Manin2002},
\cite{Katok2003}, and~\cite{Karpenkov2010}).

\vspace{2mm}

We consider the following family of {\it reduced forms}.
For every $\alpha\ge 1$ and $1> \beta \ge 0$ set
\[
f_{\alpha,\beta}=(y-\alpha x)(y+\beta x).
\]

Every form is multiple to some reduced form in appropriate basis of integer lattice $\Z^2$.
However such representation is not unique.
The following notion provides a complete invariant distinguishing different classes of reduced forms.

\begin{defn}
Let $\alpha\ge 1$, $1> \beta \ge 0$ and let
\[
\alpha=[a_0;a_1:\ldots] \quad \hbox{and} \quad
\beta=[0;a_{-1}:a_{-2}:\ldots]
\]
be the regular continued fractions for $\alpha$ and $\beta$.
Then the sequence
\[
(\ldots a_{-2},a_{-1},a_{0}, a_1,a_2,\ldots)
\]
is called the {\it LLS sequence} of the form $f_{\alpha,\beta}$.
\end{defn}
This sequence can be either finite or infinite from one or both sides.
The name for the LLS sequence (Lattice Length-Sine sequence) is due its lattice trigonometric properties,
e.g., see in~\cite{oleg2} and~\cite{oleg4}.

\begin{prop}\label{PropEquiv}
Two reduced forms are {\it equivalent} (i.e., multiple to each other after $\SL(2,\Z)$-change of coordinates)
if and only if they have the same LLS sequence up to shifts of sequence by $k$-elements
for some integer $k$ and a reversing of the order of a sequence.
\qed
\end{prop}

\begin{rem}
This statement follows directly from geometric properties of continued fractions.
As we do not use this statement in the proof of the results of this paper we skip the proof here.
We refer an interested reader to~\cite{oleg1}.
\end{rem}

Due to Proposition~\ref{PropEquiv} we can extend the notion of LLS-sequence to 
any binary quadratic form with positive discriminant.
\begin{defn}
Let $f$ be a binary quadratic form with positive discriminant.
The {\it LLS sequence} for $f$ is the {\it LLS sequence} 
for any reduced form $f_{\alpha,\beta}$ equivalent to $f$.
We denote it by $\LLS(f)$.
\end{defn}

\subsection{Classical Perron Identity}
\label{Classical Perron Identity}
We are coming to one of the most mysterious statements in theory of Markov minima.
It is known as the {\it Perron Identity}.

Let $f$ be a binary quadratic form with positive discriminant $\Delta (f)$.
Let also
\[
\LLS(f)=(\ldots a_{-2},a_{-1},a_{0}, a_1,a_2,\ldots)
\]
Then we have the following result by A.~Markov in~\cite{mar2}:
\[	 \frac{m(f)}{\sqrt{\Delta(f)}}=\inf\limits_{i\in\Z}\bigg(\frac{1}{a_i+[0;a_{i+1}:a_{i+2}:\ldots]+[0;a_{i-1}:a_{i-2}:\ldots]}\bigg).
\]

This result is based on the following observation.
Let $\alpha\ge 1$, $1> \beta \ge 0$ and let
\[
\alpha=[a_0;a_1:\ldots] \quad \hbox{and} \quad
\beta=[0;a_{-1}:a_{-2}:\ldots]
\]
be the regular continued fractions for $\alpha$ and $\beta$. Then
\[
f_{\alpha,\beta}(0,1)=\frac{1}{a_0+[0;a_1:a_2:\ldots]+[0;a_{-1}:a_{-2}:\ldots]}.
\]

Our goal is to investigate the lattice geometry behind this expression.
It will lead us to a more general rule relating continued fractions whose elements are arbitrary non zero real numbers, and 
the values of the corresponding binary form at any point on the plane 
(see Theorem~\ref{MainTheorem}, Corollary~\ref{MainCorollary} and Remark~\ref{RefGeometry}).

\subsection{LLS sequences for broken lines}
\label{LLS sequences for broken lines}
We start with the following general definition.
\begin{defn}
Consider a quadratic binary form $f$ with positive discriminant.
A broken line $A_0\ldots A_n$ is an {\it $f$-broken line} if
	the following conditions hold:
	\begin{itemize}
	\item{$A_0,A_n\ne O$ belong to the two distinct loci of linear factors of~$f$;}
	\item{all edges of the broken line are of positive length;}
	\item{for every $k=1,\ldots, n$ the line $A_{k-1}A_k$ does not pass through the origin.}
	\end{itemize}
	\end{defn}

Recall the definition of oriented Euclidean area for parallelograms.	
		
\begin{defn}\label{def}
Consider three points $A$, $B$, $C$ in the plane. Then the determinant for the matrix of vectors $AB$ and $AC$ is called
the {\it the oriented Euclidean area} for the parallelogram spanned by $AB$ and $AC$ and denoted by
\[
\det(AB,AC).
\]
\end{defn}

\begin{defn}
Let $\mathcal A=A_0A_1\ldots A_n$ be a broken line with $A_0,A_n\ne O$.
Then the sign function of the determinant $\det(OA_1,OA_n)$ is called the {\it signature} of $\mathcal A$ with respect to the origin and
denoted by $\sign(\mathcal A)$.
\end{defn}

We conclude this section with the following important definition.

\begin{defn}\label{defLLS}
Given an $f$-broken line $\mathcal A =A_0\ldots A_n$ define
	\[
	\begin{aligned}
	a_{2k}&=\det(OA_k,OA_{k+1}), \quad k=0,\ldots, n;\\
	a_{2k-1}&=\frac{\displaystyle \det(A_kA_{k-1},A_{k}A_{k+1})}{\displaystyle a_{2k-2}a_{2k}} ,
	\quad k=1,\ldots, n.\\
	\end{aligned}
	\]
	The sequence $(a_0,\ldots, a_{2n})$ is called the {\it LLS
	sequence} for the broken line and denoted by $\LLS(\mathcal A)$. 

	The expression $[a_0;\ldots: a_{2n}]$ is said to be the {\it
	continued fraction for the broken line $A_0\ldots
	A_n$}. Note that the values $a_i\ne 0$ may be negative.
	\end{defn}

The LLS sequence encodes the integer angles and integer lengths of the broken line (see~\cite{oleg1}
for further details).

\section{Generalized Perron Identity for finite broken lines}\label{thm sec}
	
Now we are in position to formulate and to prove the main result of this paper.
	
\begin{thm}\label{MainTheorem}{\bf(Generalized Perron Identity: case of finite broken lines.)}
Consider a binary quadratic form with positive discriminant $f$.
Let $\mathcal A=A_0\ldots A_{n+m}$ be an $f$-broken line $($here $n$ and $m$ are arbitrary positive integers$)$, and let
\[
\LLS(\mathcal A)=(a_0,a_1,\ldots,a_{2n+2m}).
\]
Then
\begin{equation}\label{eq1}
f(A_n)=\frac{\sign(\mathcal A) \cdot \sqrt{\Delta(f)}}{a_{2n-1}+[0;a_{2n-2}:\ldots:a_0]+[0;a_{2n}:\ldots:a_{2n+2m}]}.
\end{equation}		
\end{thm}

Let us first consider the following example.
\begin{ex}
Consider the following binary quadratic form
\[
f(x,y)=(x+y)(x-2y).
\]

\begin{figure}
\[\includegraphics[]{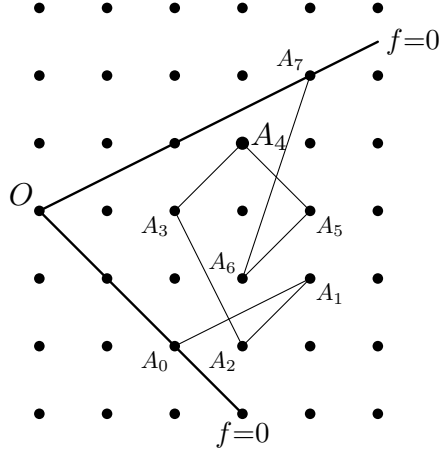}\]
\caption{The kernel of $f$ and the $f$-broken line $\mathcal A$.}\label{ex 2}
\end{figure}			

Let $\mathcal A=A_0\ldots A_7$ be the broken line with vertices
\[
\begin{array}{llll}
A_0=(2,-2), &
A_1=(4,-1), &
A_2=(3,-2), &
A_3=(2,0), \\
A_4=(3,1), &
A_5=(4,0), &
A_6=(3,-1), &
A_7=(4,2),
\end{array}
\]
see Figure~\ref{ex 2}.
Let us check Theorem~\ref{MainTheorem} for the broken line $\mathcal A$ at point $A_4=(3,1)$.
We leave the computations of LLS-sequences to a reader as an exercise, the result is as follows:
\[
\LLS(\mathcal A)=\bigg(6, -\frac{1}{30},  -5,  -\frac{3}{20},  4, \frac{3}{8},    2, -\frac14,-4: \frac{1}{8}: -4: -\frac{1}{20}: 10  \bigg)
\]
(here we denote the elements of $\LLS(\mathcal A)$ by $a_0,\ldots, a_{12}$).
Finally we have $\Delta(f)=9$ and $\sign(\mathcal A)=1$.

According to Theorem~\ref{MainTheorem} we expect the following.
\[
\begin{array}{l}
f(A_4)=
\displaystyle
\frac{\sign(\mathcal A) \cdot \sqrt{\Delta(f)}}{a_{7}+[0;a_{6}:\ldots:a_0]+[0;a_{8}:\ldots:a_{12}]}\\
\displaystyle
=
\frac{1 \cdot 3}
{-\frac{1}{4}+\big[0; 2: \frac{3}{8}: 4: -\frac{3}{20}: -5: -\frac{1}{30}: 6\big]+
\big[0; -4: \frac{1}{8}: -4: -\frac{1}{20}: 10\big]}\\
=4.
\end{array}
\]

Indeed, direct computation shows that
\[
f(A_4)=(3+1)(3-2\cdot 1)=4.
\]

\end{ex}
		
\vspace{2mm}

We start the proof with three lemmas.		

\begin{lem}\label{l1}
Consider a binary quadratic form with positive discriminant $f$.
Let $P\ne O$ and $Q\ne O$ annulate distinct linear factors of $f$.
Then for every point $A$ it holds
\[
f(A)=\sign(POQ)\cdot \frac{\det(OP,OA)\cdot \det(OA,OQ)}{\det(OP,OQ)}\cdot \sqrt{\Delta (f)}.
\]
\end{lem}

\begin{ex}
Consider the following binary quadratic form
\[
f(x,y)=(x+y)(x-2y).
\]
\begin{figure}
\[\includegraphics[]{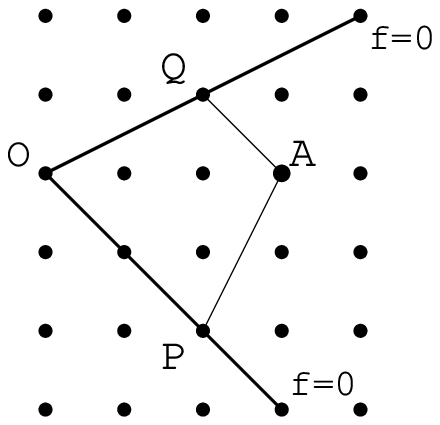}\]
\caption{The kernel of $f$ and the $f$-broken line $PAQ$.}\label{ex 1}
\end{figure}			
			
Let $PAQ$ be an $f$-broken line, with $P=(2,1)$, $A=(3,0)$, and $Q=(2,-2)$, see Figure~\ref{ex 1}.
Direct calculations show that
\[
\begin{array}{lll}
\det(OP,OA)=6,\qquad&
\det(OA,OQ)=3,\qquad&
\det(OP,OQ)=6,
\\
\sign(POQ)=1, \qquad &
f(A)=9,\qquad &
\Delta (f)=9.
\end{array}
\]
Therefore, we have
\[
\begin{aligned}
\sign(POQ)\cdot \frac{\det(OP,OA)\cdot \det(OA,OQ)}{\det(OP,OQ)}\cdot \sqrt{\Delta (f)}&=
1\cdot \frac{6\cdot 3}{6}\cdot \sqrt{9}=
9\\
&=f(A).
\end{aligned}
\]
\end{ex}

\begin{proof}[Proof of Lemma~\ref{l1}]
The statement is straightforward for the form
\[
f_\alpha(x,y)=\alpha xy.
\]
Assume that $P=(p,0)$, $Q=(0,q)$, and $A=(x,y)$.
Then we have
\[
f_\alpha(A)=\alpha xy=
\frac{py\cdot qx}{pq}\cdot \alpha
=\frac{\det(OP,OA)\cdot \det(OA,OQ)}{\det(OP,OQ)}\cdot \sqrt{\Delta(f)}.
\]
For $P=(0,p)$ and $Q=(q,0)$ we have
\[
\begin{aligned}
f_\alpha(A)&=\alpha xy=
\frac{(-px)\cdot (-qx)}{-pq}\cdot \alpha\\
&=-\frac{\det(OP,OA)\cdot \det(OA,OQ)}{\det(OP,OQ)}\cdot \sqrt{\Delta(f)}.
\end{aligned}
\]
This conclude the proof for the case of $f_\alpha$.

\vspace{2mm}

The general case follows from the invariance of the expressions of the equality of the lemma
under the group of linear area preserving  transformations (i.e., whose determinants equal 1) of the plane.
\end{proof}

Now we prove a particular case of Theorem \ref{MainTheorem}.
\begin{lem} \label{main lem}
Let $f$ be a binary quadratic form with positive discriminant.
Consider an oriented $f$-broken line $\mathcal B=B_0B_1B_2$ with
$\LLS(\mathcal B)=(b_0,b_1,b_2)$.
Then
\[
f(B_1)=\frac{\sign(\mathcal B)\cdot\sqrt{\Delta(f)}}{b_1+[0;b_0]+[0;b_2]}.
\]
\end{lem}

\begin{proof}
Set $B_i=(x_i,y_i)$ for $i=0,1,2$.
Then Definition~\ref{defLLS} implies
\[
\begin{aligned}
b_0&=\det(OB_0, OB_1)=x_0y_1-x_1y_0,\\
b_2&=\det(OB_1,OB_2)=x_1y_2-y_1x_2,\\
b_1&=\frac{\det(B_1B_0,B_1B_2)}{b_0b_2}=\frac{x_0y_2-x_2y_0-x_0y_1+x_1y_0-x_1y_2+y_1x_2}{b_0b_2}.
		\end{aligned}
\]
After a substitution and simplification we get
\[
\begin{aligned}
\frac{1}{b_1+[0;b_0]+[0;b_2]}&=&
\frac{(x_0y_1-x_1y_0)(x_1y_2-y_1x_2)}{x_0y_2-x_2y_0}\\
&=&
\frac{\det(OB_0,OB_1)\cdot\det(OB_1,OB_2)}
{\det(OB_1,OB_2)}.
\end{aligned}
\]
Finally recall that
\[
\sign(\mathcal B)=\sign(B_0B_1B_2).
\]
Now Lemma~\ref{main lem} follows directly from Lemma~\ref{l1}.
\end{proof}

For the proof of general case we need the following important result.
\begin{lem}{\bf (\cite[Corollary~$11.11$, p.~144]{oleg1}.)}\label{geometry2}
	Consider a broken line $A_0\ldots A_n$ that has the LLS sequence $(a_0,\ldots,a_{2n})$, with $A_0=(1,0)$,
	$A_1=(1,a_0)$, and $A_n=(x,y)$. Let
	\[
	\alpha=[a_0;a_1:\ldots:a_{2n}]
	\]
	be the corresponding continued fraction for this broken
	line. Then \[
	\frac{y}{x}=\alpha.
	\]
	For the case of an infinite value for $\alpha=[a_0;a_1:\ldots:
	a_{2n}]$,
	 \[
	 \frac{x}{y}=0.\]
\qed
\end{lem}

For a proof of Lemma~\ref{geometry2} we refer to~\cite{oleg1}.
As a consequence of Lemma~\ref{geometry2} we have the following statement.

\begin{cor}\label{CorGeom}
Consider two broken lines $A_0\ldots A_n$ and $B_0\ldots B_m$ that have the LLS sequences
$(a_0,\ldots,a_{2n})$, and $(b_0,\ldots,b_{2m})$ respectively.
Suppose that the following hold:
\begin{itemize}
\item $A_0=B_0$;
\item the points $A_n$, $B_m$, and $O$ are in a line;
\item the points $A_0=B_0$, $A_1$, and $B_1$ are in a line.
\end{itemize}
Then
\[
[a_0;a_1:\ldots:a_{2n}]=[b_0;b_1:\ldots:b_{2n}].
\]
\end{cor}

\begin{proof}
In coordinates of the basis
\[
e_1=OA_0, \qquad e_2=\frac{A_0A_1}{|A_0A_1||OA_0|}
\]
the coincidence of continued fractions follows from Lemma~\ref{geometry2}.
\end{proof}

\begin{proof}[Proof of Theorem \ref{MainTheorem}]
Let $f$ be a binary quadratic form with positive discriminant.
Denote the linear factors of $f$ by $f_1$ and $f_2$.
Consider an $f$-broken line $\mathcal A= A_0\ldots A_{n+m}$.
Without loss of generality we assume that $A_0$ and $A_{n+m}$ annulate $f_1$ and $f_2$ respectively.

\vspace{2mm}

\begin{figure}
\[\includegraphics[]{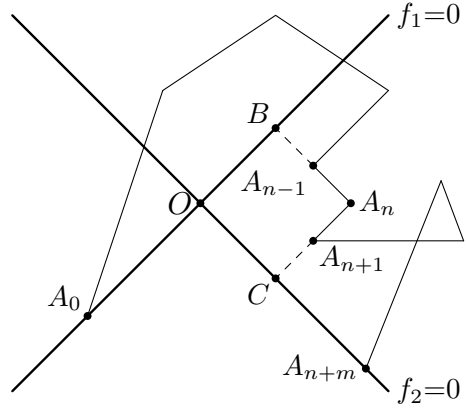}\]
\caption{The original $f$-broken line $\mathcal A$ and the resulting $f$-broken line $BA_nC$.}\label{ex 3}
\end{figure}			

Denote by $B$ the intersection of the line $A_nA_{n-1}$ with the line $f_1=0$.
Denote by $C$ the intersection of the line $A_nA_{n+1}$ with the line $f_2=0$. (See Figure~\ref{ex 3}.)
Then the continued fraction for the broken line $BA_nC$ is $[b_0:a_{2n-1}:b_2]$ for some real numbers $b_0$ and $b_2$.

By Corollary~\ref{CorGeom} we have
\[
\begin{array}{l}
[b_0]=[a_{2n-2};\ldots:a_0];\\
{[}b_2]=[a_{2n};\ldots:a_{2n+2m}].\\
\end{array}
\]
By construction
\[
\sign(BA_nC)=\sign(\mathcal A).
\]
Therefore by Lemma~\ref{main lem} we have
\[
\begin{aligned}
f(A_n)&=\frac{\sign(BA_nC)\cdot\sqrt{\Delta(f)}}{a_{2n-1}+[0;b_0]+[0;b_2]}\\
&=
\frac{\sign(\mathcal A)\cdot \sqrt{\Delta(f)}}
{a_{2n-1}+[0;a_{2n-2}:\ldots:a_0]+[0;a_{2n}:\ldots:a_{2n+2m}]}.
\end{aligned}
\]
This concludes the proof of Theorem~\ref{MainTheorem}.
\end{proof}

\section{Generalized Perron identity for asymptotic infinite broken lines}
\label{Generalized Perron identity for asymptotic infinite broken lines}

In this section we extend the Generalized Perron Identity (of Theorem~\ref{MainTheorem}) to the case of certain infinite broken lines and discuss the relation to the classical Perron Identity.

\vspace{2mm}

We start with the following definition.

\begin{defn}
Consider a binary quadratic form $f$ with positive discriminant.
An infinite in both sides broken line $\ldots A_{-2}A_{-1}A_0A_1A_2 \ldots$ is an {\it asymptotic $f$-broken line} if
the following conditions hold (here we assume that $A_k=(x_k,y_k)$ for every integer $k$):
\begin{itemize}
\item{the two side infinite sequence $\big(\frac{y_n}{x_n}\big)$ converges to different slopes of the linear factors in the kernel of $f$ as $n$ increases and decreases respectively;}
\item{all edges of the broken line are of positive length;}
\item{for every admissible $k$ the line $A_{k-1}A_k$ does not pass through the origin.}
\end{itemize}
\end{defn}

\begin{rem}
Here and below one can consider one side infinite broken lines. All the proofs are similar, so we leave 
this case as an exercise.
\end{rem}

The signature of an asymptotic $f$-broken line is defined as a determinant for two vectors in the kernel of $f$, the first with the starting limit direction
and the second with the end limit direction.

\vspace{2mm}

Finally we have a definition of LLS-sequences similar to Definition~\ref{defLLS}.
\begin{defn}\label{defLLS2}
Given an asymptotic $f$-broken line 
\[\mathcal A =\ldots A_{-2}A_{-1}A_0A_1A_2 \ldots\] 
define
	\[
	\begin{aligned}
	a_{2k}&=\det(OA_k,OA_{k+1}), \quad k\in\Z;\\
	a_{2k-1}&=\frac{\displaystyle \det(A_kA_{k-1},A_{k}A_{k+1})}{\displaystyle a_{2k-2}a_{2k}} ,
	\quad k\in \Z.\\
	\end{aligned}
	\]
	The sequence $(\ldots, a_{-2},a_{-1},a_0, a_1,a_2\ldots)$ is called the {\it LLS
	sequence} for the broken line and denoted by $\LLS(\mathcal A)$.
\end{defn}

Let us extend the Generalized Perron Identity (of Theorem~\ref{MainTheorem}) to the case of asymptotic $f$-broken line.

\begin{cor}\label{MainCorollary}{\bf(Generalized Perron Identity: case of infinite broken lines.)}
Consider a binary quadratic form with positive discriminant $f$.
Let 
\[
\mathcal A=\ldots A_{-2}A_{-1}A_0A_1A_2\ldots 
\] 
be an asymptotic $f$-broken line, and let
\[
\LLS(\mathcal A)=(\ldots a_{-2},a_{-1},a_0,a_1,a_2,\ldots).
\]
Assume also that both continued fractions
\[
[0;a_{-1}:a_{-2}:\ldots]\quad \hbox{and} \quad [0;a_1:a_2:\ldots]
\]
converge.
Then
\begin{equation}
f(A_0)=\frac{\sign(\mathcal A) \cdot \sqrt{\Delta(f)}}{a_{0}+[0;a_{-1}:a_{-2}:\ldots]+[0;a_1:a_2:\ldots]}.
\end{equation}		
\end{cor}

\begin{proof}
Without loss of generality we consider the form
\[
f=\lambda f_{\alpha,\beta}=\lambda(y-\alpha x)(y+\beta x),
\]
for some nonzero $\lambda$ and arbitrary $\alpha\ne \beta$.

\vspace{2mm}

Let $\mathcal A=\ldots A_{-2}A_{-1}A_0A_1A_2\ldots $ be an asymptotic $f$-broken line,
where $A_k=(x_k,y_k)$ for all integer $k$.
Also we assume that $x_k\ne 0$ for all $k$
(otherwise, switch to another coordinate system, where the last condition holds).

\vspace{2mm}

Set
\[
\begin{array}{c}
\mathcal A_n=A_{-n}\ldots A_{-2}A_{-1}A_0A_1A_2\ldots A_n;\\
\displaystyle
\alpha_n=\frac{y_{-n}}{x_{-n}}; \qquad \beta=\frac{y_{n}}{x_{n}}.
\end{array}
\]

First of all, by definition $\LLS(\mathcal A_n)$ coincides with $\LLS(\mathcal A)$ for all admissible entries.

\vspace{2mm}

Secondly, we immediately have that
\[
\lim\limits_{n\to \infty} \lambda f_{\alpha_n,\beta_n}(A_0)=\lambda f_{\alpha,\beta}(A_0).
\]

Thirdly, the sequence of signatures stabilizes as $n$ tends to infinity. In other words
\[
\lim\limits_{n\to \infty} \sign(\mathcal A_n)=\sign(\mathcal A).
\]

Fourthly,
\[
\lim\limits_{n\to \infty} \Delta(\lambda f_{\alpha_n,\beta_n})=\Delta(\lambda f_{\alpha,\beta}).
\]

Finally since both continued fractions
\[
[0;a_{-1}:a_{-2}:\ldots],\quad \hbox{and} \quad [0;a_1:a_2:\ldots]
\]
converge and by the above four observations we have
\[
\begin{aligned}
f(A_0)&=
\lim\limits_{n\to \infty}\lambda f_{\alpha_n,\beta_n}(A_0)
\\
&=
\lim\limits_{n\to \infty}
\frac{\sign(\mathcal A_n) \cdot \sqrt{\Delta(\lambda f_{\alpha_n,\beta_n})}}{a_{0}+[0;a_{-1}:a_{-2}:\ldots:a_{2-2n}]+[0;a_1:a_2:\ldots:a_{2n-2}]}.
\\
&=
\frac{\sign(\mathcal A) \cdot \sqrt{\Delta(f)}}{a_{0}+[0;a_{-1}:a_{-2}:\ldots]+[0;a_1:a_2:\ldots]}.
\end{aligned}
\]
The second equality holds as it holds for the elements in the limits for every positive integer $n$
by Theorem~\ref{MainTheorem}.

This concludes the proof of the corollary.		
\end{proof}
	
We conclude this paper with the following important remark.

\begin{rem}\label{RefGeometry}{\bf (Lattice geometry of the Perron Identity.)}
Let $f$ be a binary quadratic form with positive discriminant.
Consider an angle in the complement to the kernel of $f$.
The {\it sail} of this angle is the boundary of the convex hull of all integer points inside the angle except the origin. Note that each form $f$ has four angles in the complement,
and, therefore, it has four sails.

\vspace{1mm}

It is important that the sail of any angle in the complement to the set $f=0$ is an asymptotic $f$-broken line,
so the Corollary~\ref{MainCorollary} holds for each of four sails of $f$.
From general theory of geometric continued fractions, the Markov minimum is an
accumulation point of the values at vertices of all sails.

\vspace{1mm}

For every vertex $V_i$ of a sail there exists a reduced form $f_{\alpha_i,\beta_i}$.
with $\alpha_i\ge 1$ and $1\ge \beta_i > 1$ such that $V_i$ corresponds to $(0,1)$.
In particular we have 
\[
f(V_i)=f_{\alpha_i,\beta_i}(0,1).
\]
The point $(0,1)$ is a vertex of the sail for $f_{\alpha_i,\beta_i}$.
Then from general theory of continued fractions (see Part 1 of~\cite{oleg1})
the sequence $\LLS(f_{\alpha,\beta})$ coincides with the $\LLS$ sequence for the sail containing $(0,1)$.

\vspace{1mm}

Hence the expressions in the Perron Identity~(\ref{per eq}) for which the minimum is computed, i.e.,
\[		
\frac{\sqrt{\Delta(f)}}{a_i+[0;a_{i+1}:a_{i+2}:\ldots]+[0;a_{i-1}:a_{i-2}:\ldots]}
\]
for $i=\ldots, -2,-1,0,1,2,\ldots$
correspond to the formula of Corollary~\ref{MainCorollary} 
for vertices $V_i$ of all four sails. 
We consider the vertex $V_i$, with the sail containing it, as an $f$-broken line.
\end{rem}

	\bibliographystyle{plain}
	\bibliography{biblio}

\vspace{.5cm}

\end{document}